\documentclass[11pt]{amsart}
\usepackage{amssymb}
\usepackage{amsthm}

\usepackage{epsfig}
\usepackage{color}
\usepackage{verbatim}

\usepackage{tikz}

\usepackage{amsmath}
\usepackage{hyperref}
\usepackage{xypic}
\usepackage{amscd}
\usepackage{color}

\newfont{\sheaf}{eusm10 scaled\magstep1}

\newtheorem{thm}{Theorem}[section]
\newtheorem{cor}[thm]{Corollary}
\newtheorem{lemma}[thm]{Lemma}

\newtheorem{defn}[thm]{Definition}
\newtheorem{claim}[thm]{Claim}

\theoremstyle{definition}

 \newtheorem{example}[thm]{Example}

\def\length{\operatorname{length}}
\def\c1{\operatorname{c_1}}
\def\c2{\operatorname{c_2}}

\def\Sym{\operatorname{Sym}}

\def\e{\mathbf{e}}

\def\c{\mathbf{c}}

\def\PP{{\mathbb P}}

\def\H{{\mathcal H}}

\def\+{\oplus}                   
\def\*{\otimes}                  

\def\id{\operatorname{id}}

\def\bP{{\mathbb P}}

\makeatletter
\@namedef{subjclassname@2020}{\textup{2020} Mathematics Subject Classification}
\makeatother

\begin{document}

\title[On the existence problem]{On the Hurwitz existence problem for branched covers of the projective line}

\author[C.~Ciliberto]{Ciro Ciliberto}
\address{C.~Ciliberto, Dipartimento di Matematica,  Universit\`a di Roma ``Tor Vergata'', Via della Ricerca Scientifica, 00173 Roma, Italy}
\email{cilibert@axp.mat.uniroma2.it}

\author[A.~L.~Knutsen]{Andreas Leopold Knutsen}
\address{A.~L.~Knutsen, Department of Mathematics, University of Bergen,
Postboks 7800,
5020 Bergen, Norway}
\email{andreas.knutsen@math.uib.no}

\author[S.~Torelli]{Sara Torelli}
\address{S.~Torelli, Dipartimento di Matematica, Politecnico di Milano, Piazza Leonardo da Vinci 12, 20133 Milano, Italy}
\email{sara.torelli7@gmail.com}

\begin{abstract}
We give an alternative proof of the Hurwitz existence problem for branched covers of $\PP^1$ in the case where the number of ramification points equals the number of branch points, that is, where all the ramification profiles are of the form $[e,1,\ldots,1]$ with $e \geq 2$. 
\end{abstract}

\maketitle

To a branched cover $f:C \to B$ of degree $d$ between compact Riemann surfaces of genera $g=g(C)$ and $b=g(B)$ one may associate the combinatorial datum of $g,b,d$ plus the number $n$ of branch points and the $n$ partitions $\pi_1,\ldots,\pi_n$ of $d$ given by the local degrees of $f$ at the preimages of the branch points. Precisely, for each branch point $y_i \in B$, $i \in \{1,\ldots,n\}$, let $f^{-1}(y_i)=\{x_{i,1},\ldots,x_{i,\ell_i}\}$,  with $x_{i,1},\ldots,x_{i,\ell_i}$ all distinct; then 
we may write
\[ f^*y_i=\displaystyle\sum_{j=1}^{\ell_i} e_{i,j} x_{i,j}, \]
 for positive integers $e_{i,j}$ such that $\displaystyle\sum_{j=1}^{\ell_i}e_{i,j}=d$. Then each $x_{i,j}$ with $e_{i,j} \geq 2$ is called a {\it ramification point of $f$}, with {\it ramification order} $e_{i,j}$
 and {\it ramification index} $e_{i,j}-1$. The induced partition of $d$ over  $y_i$ is
 \[ \pi_i:=[e_{i,1},\ldots, e_{i,\ell_i}]\]
 and this is also called the {\it ramification profile over $y_i$}.

The data above is subject to the {\it Riemann-Hurwitz formula} 
\begin{equation}
  \label{eq:RH_cond} 2(g-1)=2d(b-1)+ \displaystyle\sum_{i=1}^n\displaystyle\sum_{j=1}^{\ell_i}(e_{i,j}-1). 
\end{equation}
An old question of Hurwitz \cite{Hur} asks which combinatorial data satisfying \eqref{eq:RH_cond} actually do occur for a branched cover, and if so, how many covers with such data exist.

If $b>0$ it is known that, for any fixed $B$, branched covers satisfying a given data do indeed exist as soon as the Riemann-Hurwitz condition \eqref{eq:RH_cond} is satisfied \cite[Prop. 3.4]{EKS} (see also \cite[Thm.1.11]{Pe}). In the case $b=0$, that is, when $B=\PP^1$, a similar result is no longer true. There are data satisfying the Riemann-Hurwitz condition \eqref{eq:RH_cond} for which no branched cover exists (see Example \ref{ex:noex} below). 
Many partial answers to the Hurwitz existence problem have been given, but a complete solution is still lacking. We refer to \cite{Pe} for an exposition of the problem and known results, with several examples, also in the non-compact (and thus non-orientable) case.

In this note we consider the case of covers of $B=\PP^1$ where the number of ramification points equals the number of branch points, that is, when all partitions are of the form
\[ \pi_i=[e_i,\underbrace{1,\ldots,1}_\text{$d-e_i$ times}]=:[e_i,1^{d-e_i}], \;\; 2 \leq e_i \leq d.\]
The Riemann-Hurwitz condition \eqref{eq:RH_cond} in this case reads
\begin{equation}
  \label{eq:RH_cond2} \displaystyle\sum_{i=1}^ne_i=2(g-1+d)+n.
\end{equation}

It turns out that in this case the Riemann-Hurwitz condition is {\it sufficient} for the existence of a branched cover:

\begin{thm} \label{thm:esistenza}
Let $d \geq 2$,  $n \geq 1$  and  $g \geq 0$  be integers, and $\mathbf{e}=(e_1,\ldots,e_n)$ an $n$-tuple of integers such that $2 \leq e_i \leq d$ for all $i \in \{1,\ldots,n\}$ and such that the Riemann-Hurwitz condition \eqref{eq:RH_cond2} holds.

  Let $y_1,\ldots,y_n \in \PP^1$ be distinct points. Then there exists a compact Riemann surface $C$ of genus $g$ and a cover $f:C \to \PP^1$ branched at $y_1,\ldots, y_n$ with ramification profile $[e_i,1^{d-e_i}]$ over $y_i$.
  \end{thm}

  This result is a consequence of \cite[Cor. F]{CFG}, and in this note we give an alternative proof. Let us also mention that the following special cases of the theorem were known earlier: If $\pi_1=[d]$ or $\pi_1=[1,d-1]$, $n=1$ or more in general if $\sum_{i=1}^n(e_i-1)\geq 3(d-1)$ as proven in \cite{EKS}; if $n=3$ and $\pi_3=[k,1,\cdots,1]$ as proven in \cite{B}; if the domain is $\bP^1$ as proven in \cite{B,SX}; if $\sum_{i=1}^n(e_i-1) \leq 2(d-1)+g$ as proven in \cite{Os}. Our proof is independent of the aforementioned results.

  We state here an interesting consequence regarding {\it nonemptiness} of Hurwitz spaces.

Consider, for any integers $d \geq 2$, $n,g \geq 0$, and any $n$-tuple of integers $e_i$ such that \linebreak $2 \leq e_i \leq d$, the {\it Hurwitz space} \[ \H_{g,k;\e} \]
parametrizing pairs
\[
  \{(C,x_1,\ldots,x_n),f:C \to \PP^1\},
\]
where $C$ is a smooth curve of genus $g$, and $x_1,\ldots, x_n \in C$ are distinct marked points, and $f:C \to \PP^1$ is a morphism of degree $d$ having ramification order $e_i$ at $x_i$, and simple ramification at $r:=2(g-1-k(b-1))+n-e$ points completing the ramification profile. 

For the general member in any component of  $ \H_{g,k;\e}$, using  the {\it Riemann Existence Theorem} (cf., e.g., \cite[III, Cor. 4.10]{Mi} or \cite[Thm. 1.8]{Pe}), one sees that the number of ramification points equals the number of branch points, which means that the cover $f:C \to \PP^1$ is branched over
$n+r$ points with ramification profile $[e_i,1^{d-e_i}]$ over $y_i:=f(x_i)$, $i \in \{1,\ldots,n\}$, and ramification profile $[2,1^{d-2}]$ over a remaining set of unmarked branch points. The Riemann-Hurwitz condition \eqref{eq:RH_cond} is of course a {\it necessary} condition for the nonemptiness of $\H_{g,k;\e}$, and it now reads
\begin{equation}
  \label{eq:RH_cond3} \displaystyle\sum_{i=1}^ne_i \leq 2[g-1-d(b-1)]+n,
\end{equation}
with the {\it inequality} being due to the fact that we have not fixed all ramification points.

As an immediate consequence of Theorem \ref{thm:esistenza},  one can settle the nonemptiness problem regarding Hurwitz schemes over $\PP^1$:

\begin{cor} \label{cor:hur}
  The Hurwitz moduli space $\H_{g,k;\e}$ is
nonempty if and only if  the Riemann-Hurwitz condition
\eqref{eq:RH_cond3} holds.
\end{cor}

\vspace{0.2cm} {\it Acknowledgments.} The authors thank Margherita Lelli-Chiesa and Carlo Petronio for enlightning conversations and correspondence on the matter. A.L.K. was partially supported 
	by the Trond Mohn Foundation's project ``Pure Mathematics in
	Norway''. C.C. and S.T. are members of GNSAGA of the Istituto Nazionale di Alta Matematica ``F. Severi''. In the first version of the note posted on arXiV we claimed that the Hurwitz existence problem was not known in the case we consider here; we would like to thank Gianluca Faraco for pointing out to us
that the result had been proved in \cite{CFG}.

\section{Proof of the main theorem} \label{sec:proof}
  
 \subsection{The Riemann Existence Theorem}
 
The proof is based on the Riemann Existence Theorem,
which states that there is a one-to-one correspondence between degree $d$-covers 
of $\PP^1$ branched at
$n$ points $y_1,\ldots,y_n$ of $\PP^1$ with ramification profile
\[\pi_i:=[e_{i,1},\ldots, e_{i,\ell_i}] \]
over $y_i$ and conjugacy classes of $n$-tuples $(\sigma_1,\ldots,\sigma_n)$ of permutations in the symmetric group $\Sym(d)$ satisfying:
\begin{itemize}
\item $\sigma_i$ has cycle structure $\pi_i$,
\item  the subgroup generated by the $\sigma_i$s is transitive,
  \item $\sigma_1\cdots \sigma_n=\id_d$ (the identity element in $\Sym(d)$).  
\end{itemize}

Before starting the proof, it can be instructive to look at the following example mentioned in the introduction:

\begin{example}{\rm (\cite[Prop. 1.7 or Prop. 1.10]{Pe})} \label{ex:noex}
  Take $d=4$, an integer $n \geq 3$, and the partitions
  \[ \pi_1=\cdots=\pi_{n-1}=[2,2], \;\; \pi_n=[3,1].\]
  Let $g:=n-3$. Then the Riemann-Hurwitz formula \eqref{eq:RH_cond} is satisfied. However, there is no cover $C \to \PP^1$ from a genus $g$ Riemann surface $C$ with ramification profile $[\pi_1,\ldots,\pi_n]$ over $n$ points of $\PP^1$. Indeed, if it existed, there would exist $\sigma_1,\ldots,\sigma_{n-1} \in \Sym(4)$ with cyclic structure $[2,2]$ such that $\sigma_1\cdots\sigma_{n-1}$ has cyclic structure $[3,1]$. However, any two permutations  with cyclic structure $[2,2]$ are, up to conjugation, $(1,2)(3,4)$ or $(1,3)(2,4)$, so their product is either the identity or $(1,4)(2,3)$. Hence, the product of permutations  with cyclic structure $[2,2]$ can never be of cyclic structure $[3,1]$.
\end{example}

Next we will start the proof of the main result. 
We use the following conventions:
\begin{itemize}
\item The order of multiplication  of two permutations   $\sigma, \gamma \in \Sym(d)$  is defined by $ (\sigma \gamma)(x)=\sigma(\gamma(x))$.
\item  A {\it cycle (of length} or {\it order $l$)} in $\Sym(d)$ is a permutation with cyclic structure $[l,1,\ldots,1]$, with $l \leq d$, that is, a permutation that can be written in the form $(a_1\; a_2\; \cdots\; a_l)$, for $l$ distinct integers $a_1,\ldots, a_l\in \{1,\ldots,d\}$. Occasionally, and by some abuse of notation, we can consider such a cycle as an element in $\Sym(m)$ for any $m \geq l$, acting as the identity on $m-l$ elements, and cyclically on the rest.
 \item The identity $\id_d \in \Sym(d)$ is considered to be a cycle of  length $1$, as it may be written as  $\id_d=(x)$ for any $x \in \{1,\ldots,d\}$.
\end{itemize}

\subsection{Augmentations}

 We start with the following observation: If  two cycles  $\sigma_1,\sigma_2 \in \Sym(d-1)$ are not disjoint\footnote{This means that there is an element moved by both permutations}, or at least one of them is the identity, we may write
\begin{equation} \label{eq:aug0}
  \sigma_1=(x \;\; a_1 \cdots a_s) \;\; \mbox{and} \;\; \sigma_2=(b_1 \cdots b_t \;\; x),
  \end{equation}
for some $x$, where $s:=\length \sigma_1-1$ and $t:=\length \sigma_2-1$.
  Set
  \begin{eqnarray} \label{eq:aug1}
    \widetilde{\sigma}_1 & := &  (x\;d) \sigma_1= (x \;\; a_1 \cdots a_s\;\;d), \\
    \nonumber   \widetilde{\sigma}_2 & := & \sigma_2 (x\;d) = (d\;\;b_1 \cdots b_t \;\; x).
\end{eqnarray}
Then, viewing $\sigma_1$  and $\sigma_2$ as elements of $\Sym(d)$, one has 
\begin{equation} \label{eq:prodo}
  \sigma_2 \sigma_1= \widetilde{\sigma}_2\widetilde{\sigma}_1.
  \end{equation}

  We can extend this construction further: Assume that we have two cycles $\sigma_1, \sigma_2 \in \Sym(d-1)$ and an element $\alpha \in \Sym(d-1)$ such that
  \begin{eqnarray} \label{eq:cond-aug}
    & \sigma_2\;\; \mbox{and} \;\; \alpha  \sigma_1  \alpha^{-1} \;\; \mbox{(equivalently,} \;\;  \alpha^{-1}  \sigma_2 \alpha \;\; \mbox{and} \;\; \sigma_1) \;\; \mbox{are not disjoint,} & \\
\nonumber    & \mbox{or at least one of them is the identity.} &
  \end{eqnarray}
  Then we may write
  \begin{equation} \label{eq:aug0+}
  \sigma_1=(\alpha^{-1}x \;\; a_1 \cdots a_s) \;\; \mbox{and} \;\; \sigma_2=(b_1 \cdots b_t \;\; x),
  \end{equation}
  with $s:=\length \sigma_1-1$ and $t:=\length \sigma_2-1$.
Set
\begin{equation} \label{eq:aug1+}
  \widetilde{\sigma}_1:= (\alpha^{-1}x\;d)\sigma_1 =  (\alpha^{-1}x \;\; a_1 \cdots a_s\;\;d) \;\; \mbox{and} \;\; \widetilde{\sigma}_2:= \sigma_2(x\;d)=    (d\;\;b_1 \cdots b_t \;\; x).  \end{equation}
  (Thus, the construction in \eqref{eq:aug0} is the special case with $\alpha=\id_{d-1}$.) Then, again viewing $\sigma_1$, $\sigma_2$ and $\alpha$ as elements of $\Sym(d)$, one may check that
\begin{equation} \label{eq:prodo+}
  \sigma_2 \alpha \sigma_1= \widetilde{\sigma}_2\alpha \widetilde{\sigma}_1.
  \end{equation}
  Again we see that \eqref{eq:prodo} is the special case with $\alpha=\id_{d-1}$. 
   
  \begin{defn} \label{def:k-aug}
    We call the pair $\{ \widetilde{\sigma}_1,\widetilde{\sigma}_2\}$ a {\em $d$-augmentation} of the pair $\{\sigma_1,\sigma_2\}$ (with respect to $\alpha$ and $x$).

    If $\left(\gamma_1,\ldots,\gamma_n\right)$ is an $n$-tuple of cycles in $\Sym(d-1)$, a {\em $d$-augmentation} of a pair $\{\gamma_i,\gamma_j\}$ with $i<j$ is a $d$-augmentation of the pair with respect to
    \[ \alpha=
      \begin{cases}
        \gamma_{j-1}\cdots\gamma_{i+1}, & \mbox{if} \;\; i <j-1, \\
        \id_{d-1}, & \mbox{if} \;\; i =j-1.
      \end{cases}
\]
  \end{defn}

  Note that a  $d$-augmentation is simply done by inserting a $d$ in a suitable spot in each cycle.

 For a $d$-augmentation to be defined, condition \eqref{eq:cond-aug} must be satisfied. This condition is always satisfied if one among the two cycles  has  maximal length $d-1$ or if one is the identity. In our applications, one of the cycles will always have maximal length, or the two cycles are not disjoint or one is the identity and we use $\alpha=\id_{d-1}$.

  \begin{example}
    Assume $\sigma_1=\sigma_2=\id_{d-1}$. Then, for any $x \in \{1,\ldots,d-1\}$, we may write $\sigma_1=\sigma_2=(x)$ and thus
    \[ \widetilde{\sigma}_1:=(x\;d) \;\; \mbox{and} \;\; \widetilde{\sigma}_2:=(d\;x) \]
    is  a $d$-augmentation with respect to $x$ (and $\alpha=\id_{d-1}$).

    Assume that $\sigma_1$ is a length-$(d-1)$ cycle, say
\[ \sigma_1=(1 \; 2\; \cdots \; (d-1)),\]
 $\alpha \in \Sym(d-1)$ is an arbitrary permutation and $\sigma_2 \in  \Sym(d-1)$ is an arbitrary cycle of length $t+1$ for some $t  \geq 0$. Pick any $x \in \{1,\ldots,d-1\}$ moved by $\sigma_2$ if $\sigma_2 \neq \id_{d-1}$ and pick any $x \in \{1,\ldots,d-1\}$ if  $\sigma_2 =\id_{d-1}$. Then we may write
\[\sigma_2=(b_1\cdots b_t\; x), \]
and 
\begin{eqnarray*}
  \widetilde{\sigma}_1 & := & \left(1 \cdots \; (\alpha^{-1}x-1) \; \; d \; \; \alpha^{-1}x \cdots \; (d-1)\right) \\
   \widetilde{\sigma}_2 &  := &  (d \; b_1\cdots b_t\; x)
\end{eqnarray*}
    is  a $d$-augmentation with respect to $\alpha$ and $x$.
  \end{example}

  The following observations will be central:

  \begin{lemma} \label{lemma:K4}
    Let $\left(\gamma_1,\ldots,\gamma_n\right)$ be an $n$-tuple of cycles in $\Sym(d-1)$. Let $\left(\gamma'_1,\ldots,\gamma'_n\right)$ be an $n$-tuple of cycles obtained from it by $d$-augmenting one or more pairs of cycles, and viewing the remaining ones as cycles in $\Sym(d)$. Then
    \begin{itemize}
    \item[(i)]  $\gamma'_1 \cdots \gamma'_n=\gamma_1 \cdots \gamma_n$, viewing
      also $\gamma_1,\ldots,\gamma_n$ as cycles in $\Sym(d)$;
    \item[(ii)] if the group generated by $\{\gamma_1,\ldots,\gamma_n\}$ in $\Sym(d-1)$ is transitive, then so is the group  generated by $\{\gamma'_1,\ldots,\gamma'_n\}$ in $\Sym(d)$.
    \end{itemize}
  \end{lemma}

  \begin{proof}
    Property (i) follows from \eqref{eq:prodo+}. To prove (ii), it is sufficient to note that from \eqref{eq:aug0+} and  \eqref{eq:aug1+} one computes
    \begin{eqnarray*}
      \widetilde{\sigma}_1(\alpha^{-1}x)=\sigma_1(\alpha^{-1}x), \;\; \widetilde{\sigma}_1(a_i)=\sigma_1(a_i) \; \mbox{for} \; i \in \{1,\ldots,s-1\}, \;\; \widetilde{\sigma}_1^2(a_s)=\sigma_1(a_s), \\
    \widetilde{\sigma}_1(d)=\alpha^{-1}x, \;\;  \widetilde{\sigma}_1^{i+1}(d)=a_{i} \; \mbox{for} \; i \in \{1,\ldots,s\} 
      \end{eqnarray*}
  and
 \begin{eqnarray*}
   \widetilde{\sigma}_2^2(x)=\sigma_2(x), \;\; \widetilde{\sigma}_2(b_i)=\sigma_2(b_i) \; \mbox{for} \; i \in \{1,\ldots,t\},\\
   \widetilde{\sigma}_2^i(d)=b_{i} \; \mbox{for} \; i \in \{1,\ldots,t\}, \;\;
 \widetilde{\sigma}_2^{t+1}(d)=x.  
      \end{eqnarray*} 
\end{proof}

  \subsection{A first basic lemma}

 The next  lemma, together with Lemma \ref {lemma:K2} below, are the central ingredients in the proof of Theorem \ref{thm:esistenza}.

\begin{lemma} \label{lemma:K1}
  Fix integers $d \geq 2$ and $n \geq 2$. Let $\left(e_1,\ldots,e_n\right)$ be
an $n$-tuple of
  integers such that $2 \leq e_i \leq d$ for all $i \in \{1,\ldots,n\}$, at least one $e_i=d$, and  $\displaystyle\sum_{i=1}^n e_i-n$ is even. Then there exists an $n$-tuple of cycles $\left(\sigma_1,\ldots,\sigma_n\right)$ in $\Sym(d)$ of lengths $\length \sigma_i=e_i$ such that $\sigma_n\cdots \sigma_1=\id_d$ and
  such that, for all $i \in \{1,\ldots,n-1\}$, the cycles $\sigma_i$ and $\sigma_{i+1}$ are not disjoint.
\end{lemma}

\begin{proof}
 We set for simplicity $I=\{1,\ldots,n\}$. We first prove the lemma in two base cases. \medskip

  {\bf Base case I: $e_i=2$ for all $i \in I$.} In this case $d=2$ and $\sum_{i=1}^n e_i-n=2n-n=n$, which is even by assumption. Then we may take  $\sigma_i=(1\; 2)$ for all $i \in I$.\medskip

  {\bf Base case II: $e_i=d$ for all $i \in I$.}
If $n$ is even, there exists an $n$-tuple of length-$d$ cycles $\left(\sigma_1, \ldots, \sigma_{n}\right)$ such that
  $\sigma_n \cdots \sigma_1=\id_d$; take for instance pairwise inverse cycles in order.

  If $n$ is odd, then  $n \geq 3$, and,   as $\sum_{i=1}^n e_i-n=dn-n=n(d-1)$ is even by assumption, $d$ must be odd. As $n-1$ is even, there exist as above an $(n-1)$-tuple of length $d$ cycles $\sigma_1, \ldots, \sigma_{n-1}$ such that
  $\sigma_{n-1} \cdots \sigma_{1}=\id_d$. Since $d$ is odd, there is a length $d$ cycle $\sigma_n$ such that
  $\sigma_n^2=\sigma_{n-1}$. (This follows for instance from the observation that
    $\left(1\;2\;\cdots\;d\right)^2=\left(1\;3\;5\;\cdots\;d\;2\;4\;\cdots\; (d-1)\right)$.)
  Thus, $\sigma_n^2\sigma_{n-2} \cdots \sigma_1 =\id_d$, and
  the $n$-tuple  $\{\sigma_1,\ldots, \sigma_{n-2},\sigma_n,\sigma_n\}$ satisfies the conditions in the lemma. \medskip

We are now done with the special cases.

In general, we define $J_0:=\{i \in I \;|\; e_i=d\}$ and denote by $n_0$ its cardinality. (Note that $n_0>0$.)  
  We will prove the lemma  by recursion, with  the base cases  being  cases I-II above.
  
  We may assume that there is some $i \in I$ such that $e_i > 2$ and some
  $i \in I$ such that $e_i <d$. In particular,  $d>2$ and $n_0<n$. The latter means that $I\setminus J_0 \neq \emptyset$.

  We define a subset $J \subset \{1,\ldots,n\}$ as follows:
  \begin{itemize}
\item If $n_0 \geq 2$ and $n_0$ is even, define  $J:=J_0$.
\item If $n_0$ is odd, pick an index $\xi \in I\setminus J_0$
and define $J:=J_0 \cup \{\xi\}$. 
   \end{itemize}
Note that $J$ contains all indices $i$ for which $e_i=d$.

Let $N$ denote the cardinality of $J$, so that
\[
N =
\begin{cases}
n_0, & \mbox{if} \; n_0 \geq 2 \;\mbox{and even}, \\
  n_0+1, & \mbox{if} \; n_0 \;\; \mbox{is odd}.
\end{cases}
\]
Note that $N$ is even. 

We set
\[ e'_i:=
  \begin{cases}
    e_i-1, & \mbox{if} \; i \in J, \\
    e_i, & \mbox{if} \; i \not \in J.
  \end{cases}
\]
Then $1 \leq e'_i \leq d-1$ for all $i \in I$. Moreover,
$e'_i=1$ for at most one index $i$, which is $i=\xi$ and this occurs precisely if 
\[ n_0 \; \mbox{is odd and} \; e_{\xi}=2.\]
We set
\[
  I':=\{i\; | \; e'_i \geq 2\} =
  \begin{cases}
    I, & \mbox{if} \;\; e_{\xi}>2, \\
    I\setminus \{\xi\}, & \mbox{if} \;\; e_{\xi}=2
  \end{cases}
\]
and let $n'$ be the cardinality of $I'$.
Thus
\[
n'=
\begin{cases}
  n, & \mbox{if} \;\; e_{\xi}>2, \\
  n-1, & \mbox{if} \;\; e_{\xi}=2
\end{cases}
  \]
  and
  \[ 2 \leq e'_i \leq d-1\;\; \mbox{for all} \;\; i \in I', \;\; e'_i=d-1 \;\;
\mbox{for all} \;\; i \in J_0.\]
  Setting $e:=\displaystyle\sum_{i \in I} e_i$ and $e':=\displaystyle\sum_{i \in I'} e'_i$, we thus have
\[ e'=
  \begin{cases}
    e-N, & \mbox{if} \; n'=n, \\
    e-N-1, & \mbox{if} \; n'=n-1.
  \end{cases}
\]
Hence $e'-n'=e-n-N$, which is still even, by the assumption that $e-n$ is even and the fact that $N$ is even.

If $e_i'=d-1$ for all $i \in I'$ or $e_i'=2$ for all $i \in I'$, we are in one of the base cases above  (with respect to $I'$). Otherwise, there is some $i \in I'$ such that $e'_i > 2$ and some
  $i \in I$ such that $e'_i <d-1$, and we may repeat the procedure  with $(I',n',d-1)$ in place of $(I,n,d)$, until we reach a base case, for which the lemma holds, as proved above. The lemma is therefore proved if we can show that whenever it holds at a certain step of the procedure, it also holds for the previous step. It therefore suffices to prove the following:

\begin{claim}
 If the lemma is true for $(I',n',d-1)$, it is also true for $(I,n,d)$.  
\end{claim}

\begin{proof}[Proof of claim]
  
By  hypothesis  there exists an $n'$-tuple of cycles $\left(\sigma'_i\right)_{i \in I'}$ in $\Sym(d-1)$ of lengths $\length \sigma'_i=e'_i$ such that $\sigma'_n\cdots \sigma'_1=\id_{d-1}$ (dropping the index $\xi$ if $n'=n-1$) and such that two successive cycles are not disjoint. If $n'=n-1$, define
$\sigma'_{\xi}:=\id_{d-1}$, which has length $1=e'_{\xi}$ by convention. Then  the whole $n$-tuple of cycles $\left(\sigma'_i\right)^n_{i=1}$ in $\Sym(d-1)$ of lengths $\length \sigma'_i=e'_i$ satisfies $\sigma'_n\cdots \sigma'_1=\id_{d-1}$ and two successive nontrivial cycles in the tuple are not disjoint. All $\sigma'_i$ with $i \in J$ have maximal length $d-1$, except possibly for $\sigma'_{\xi}$. We can thus $d$-augment all $\sigma'_i$ with $i \in J$ in pairs to obtain cycles $\sigma_i$ in $\Sym(d)$; if $\sigma'_{\xi}=\id_{d-1}$, we pick
\[
x=
\begin{cases}
  \mbox{a common element in $\sigma'_{\xi-1}$ and $\sigma'_{\xi+1}$, if $1<\xi<n$},\\
 \mbox{any element in $\sigma'_{2}$, if $\xi=1$}, \\ 
 \mbox{any element in $\sigma'_{n-1}$, if $\xi=n$}, 
\end{cases}
  \]
  and we $d$-augment with respect to $x$ (that is, we write
  $\sigma'_{\xi}=(x)$ and set $\sigma_{\xi}=(x \; d)$).
  
Setting
$\sigma_i=\sigma'_i$ for all $i \in I\setminus J$, and considering them as cycles in $\Sym(d)$, we obtain, by Lemma \ref{lemma:K4}, a sequence of cycles $\{\sigma_i\}^n_{i=1}$ in $\Sym(d)$ of lengths $\length \sigma_i=e_i$ satisfying
$\sigma_n\cdots \sigma_1=\id_d$ and two successive cycles in the sequence are still not disjoint. The latter is due to the fact that the process of $d$-augmentation is simply done by inserting $d$ in a suitable spot in the cycles, and by our choice of $x$ in the case $\sigma'_{\xi}=\id_{d-1}$. 
\end{proof}

This finishes the proof of the lemma.
\end{proof}

\subsection{Admissible and extremal $n$-tuples of integers} 

We will now improve Lemma \ref {lemma:K1}. To ease language, we will make the following:

\begin{defn} \label{def:k-adm}
  Fix integers $d \geq 2$ and  $n \geq 2$.  We say that an $n$-tuple of integers $\left(e_1,\ldots,e_n\right)$ is {\em $(d,n)$-admissible} if
  \begin{equation}
    \label{eq:k-adm-1}
    2 \leq e_i \leq d \;\; \mbox{for all} \;\;i \in \{1,\ldots,n\}
  \end{equation}
and 
\begin{equation}
    \label{eq:k-adm-2}
  \displaystyle\sum_{i=1}^n e_i-n \;\; \mbox{is even  and}  \;\; \displaystyle\sum_{i=1}^n e_i-n \geq 2(d-1). 
\end{equation}

We additionally say that $\left(e_1,\ldots,e_n\right)$ is {\em extremal} if
in addition some $e_i=d$ or all $e_i=2$. 
\end{defn}

The reason for the terminology {\it extremal} is that the extremal $(d,n)$-admissible $n$-tuples are the basis cases of the  recursive  procedure in the next  proof. 

\begin{lemma} \label{lemma:K2}
  For  any    $(d,n)$-admissible $n$-tuple of integers $\left(e_1,\ldots,e_n\right)$ there is an $n$-tuple of cycles $\left(\sigma_1,\ldots,\sigma_n\right)$ in $\Sym(d)$ of lengths $\length \sigma_i=e_i$ such that
\begin{itemize}
    \item[(i)] $\sigma_n\cdots\sigma_1=\id_d$,
    \item[(ii)] the subgroup generated by $\sigma_1,\ldots,\sigma_n$ is transitive,
    \item[(iii)] for all $i \in \{1,\ldots,n-1\}$, $\sigma_i$ and $\sigma_{i+1}$ are not disjoint.
      \end{itemize}
\end{lemma}

\begin{proof}
We set  $I=\{1,\ldots,n\}$. We first prove the lemma in two base cases, corresponding to   extremal    $(d,n)$-admissible $n$-tuples.\medskip

{\bf Base case I: $e_i=2$ for all $i \in I$.} In this case $\sum_{i=1}^n e_i-n=2n-n=n$. Hence, by properties \eqref{eq:k-adm-1}-\eqref{eq:k-adm-2}, we have that $n$ is even and $n \geq 2d-2$. Consider the $2d-2$ transpositions
    \[ \sigma_j=(j\;\;j+1), \;\; \sigma_{d-1+j}=\sigma_{d-j} \;\; \mbox{for} \;\;  j\in \{1,\ldots,d-1\}.\]
    Then the $n$-tuple $\left(\sigma_j\right)_{j=1}^{2(d-1)}$ satisfies conditions (i)-(iii) in the lemma with $n=2(d-1)$. As $n$ is even, we may then for instance
set $\sigma_j:=(1\; 2)=\sigma_{2(d-1)}$ for all $j \in \{2(d-1)+1,\ldots, n\}$ to obtain an $n$-tuple $\left(\sigma_j\right)_{j=1}^{n}$ still satisfying conditions (i)-(iii).\medskip

{\bf Base case II: there is an $i \in I$ such that $e_i=d$.} In this case, by Lemma \ref{lemma:K1}, there
exists an $n$-tuple of cycles $\left(\sigma_1,\ldots,\sigma_n\right)$ in $\Sym(d)$ with $\length \sigma_i=e_i$ satisfying (i) and (iii). Property (ii) is immediate, as at least one of the cycles has length $d$.\medskip

We then proceed  recursively, reducing to one of the two base cases, in the following way. We may assume that $\left(e_1,\ldots,e_n\right)$ is non-extremal, that is,
\begin{equation}
  \label{eq:non-ext}
2 \leq e_i<d \;\; \mbox{for all} \;\;i \in I \;\;\;\; \mbox{and} \;\;\;\; e_i>2 \;\; \mbox{for some $i \in I$ \;\;\;\; (whence $d >2$)}.  
\end{equation}

Pick one $i_0\in I$ for which $e_{i_0} \geq 3$ and pick a neighboring index $i_1 \in I$ of $i_0$, that is, $i_1=i_0-1$ or $i_1=i_0+1$. 
(This means that if $i_0=1$, we take $i_1=2$, if $i_0=n$, we take $i_1=n-1$, and otherwise we have the two choices  $i_1=i_0\pm 1$.) 
Set 
    \[
  e'_i    =\begin{cases}
    e_i-1, & \; \; \mbox{for} \;\; i \in \{i_0,i_1\}, \\
    e_i, & \; \; \mbox{for} \;\; i \in I \setminus \{i_0,i_1\}.
    \end{cases}
\]
Then $e'_i \leq d-1$ for all $i \in I$ and $e'_i \geq 1$ for all $i \in I$, with
equality possibly occurring only for $i=i_1$, when $e_{i_1}=2$. We set 
\[
  I':=\{i\; | \; e'_i \geq 2\}=
  \begin{cases}
    I, & \mbox{if} \;\; e_{i_1}>2, \\
    I\setminus \{i_1\}, & \mbox{if} \;\; e_{i_1}=2
  \end{cases}
\]
and let $n'$ be the cardinality of $I'$.
Thus $n'=n$ or $n-1$. We note that in any event,
$n' \geq 2$. Indeed, $n'=1$ would imply $n=2$ and $\{e_1,e_2\}=\{s,2\}$  for some $s \geq 3$. But then assumption \eqref{eq:k-adm-2} would yield  $s=\displaystyle\sum_{i=1}^n e_i-n \geq 2(d-1)$, which is impossible, as $s<d$ and $d >2$.

Setting $e:=\displaystyle\sum_{i \in I} e_i$ and $e':=\displaystyle\sum_{i \in I'} e'_i$, we thus have
\[ e'=
  \begin{cases}
    e-2, & \mbox{if} \; n'=n, \\
    e-3, & \mbox{if} \; n'=n-1.
  \end{cases}
\]
Hence $e'-n'=e-n-2$, and  assumptions \eqref{eq:k-adm-1}--\eqref{eq:k-adm-2} yield  that
\[ e'-n' \;\; \mbox{is even and}  \;\; e'-n' \geq 2(d-2).\]
In other words, the $n'$-tuple $\left(e'_i\right)_{i \in I'}$  is     $(d-1,n')$-admissible. 

If $\left(e'_i\right)_{i \in I'}$ is extremal, that is, $e'_i=2$ for all $i \in I'$ or $e'_i=d-1$ for some $i \in I'$, we are in one of the base cases above. If not,
\begin{equation}
  \label{eq:non-ext'}
2 \leq e'_i<d-1 \;\; \mbox{for all $i\in I'$} \;\; \mbox{and} \;\; e'_i>2 \;\; \mbox{for at least one $i \in I'$ (whence $d-1 >2$)},  
\end{equation}
(cf. \eqref{eq:non-ext}) and the procedure continues  with $(I',n',d-1)$ in place of $(I,n,d)$, until we reach a base case, for which the lemma holds, as proved above. The lemma is therefore proved if we can show that whenever it holds at a certain step of the procedure, it also holds for the previous step. It therefore suffices to prove the following:

\begin{claim}
 If the lemma is true for $(I',n',d-1)$, it is also true for $(I,n,d)$.  
\end{claim}

\begin{proof}[Proof of claim]
 
By  hypothesis, there is an $n'$-tuple of cycles $\left(\sigma'_i\right)_{i \in I'}$ in $\Sym(d-1)$ of lengths $\length \sigma'_i=e'_i$ generating a transitive subgroup and such that $\sigma'_n\cdots \sigma'_1=\id_{d-1}$ (dropping the index $i_1$ if $n'=n-1$, that is, $e'_{i_1}=1$), and two successive cycles are not disjoint. If $n'=n-1$, we define $\sigma'_{i_1}:=\id_{d-1}$, so that the whole $n$-tuple of cycles $\left(\sigma'_i\right)^n_{i=1}$ in $\Sym(d-1)$ of lengths $\length \sigma'_i=e'_i$ generates a transitive subgroup and satisfies $\sigma'_n\cdots \sigma'_1=\id_{d-1}$ and two successive nontrivial cycles in the $n$-tuple are  not disjoint. 

The two indices $i_0$ and $i_1$ are successive, so by assumption the
pair $\{\sigma'_{i_0},\sigma'_{i_1}\}$ is either not disjoint, or by construction $\sigma'_{i_1}=\id_{d-1}$. In the latter case,  by assumption the cycles $\sigma'_{i_0}$ and $\sigma'_{i_0+2}$ are not disjoint, if $i_0<i_1=i_0+1<n$, and the cycles $\sigma'_{i_0-2}$ and $\sigma'_{i_0}$ are not disjoint $1<i_1=i_0-1<i_0$. We may therefore 
pick
\[
  x =
  \begin{cases}
    \mbox{any element moved by $\sigma'_{i_0}$, if $i_1=1$ or $i_1=n$}, \\
    \mbox{any element moved by both $\sigma'_{i_0}$ and $\sigma'_{i_0+2}$ if $i_0<i_1=i_0+1<n$}, \\
     \mbox{any element moved by both $\sigma'_{i_0-2}$ and $\sigma'_{i_0}$ if $1<i_1=i_0-1<i_0$}.
  \end{cases}
  \]
  
We can now $d$-augment the pair $\{\sigma'_{i_0},\sigma'_{i_1}\}$ (with respect to $x$, in the case $\sigma'_{i_1}=\id_{d-1}$)
to a pair of cycles
$\{\sigma_{i_0},\sigma_{i_1}\}$ in $\Sym(d)$. Setting
$\sigma_i=\sigma'_i$ for all $i \not \in \{i_0,i_1\}$, and considering them as cycles in $\Sym(d)$, we obtain, by Lemma \ref{lemma:K4}, an $n$-tuple of cycles $\left(\sigma_i\right)^n_{i=1}$ in $\Sym(d)$ of lengths $\length \sigma_i=e_i$ satisfying
properties (i)-(iii).  (Property (iii) follows from the fact that the process of $d$-augmentation is done by inserting $d$ in a suitable spot in the cycles, and by our choice of $x$ in the case $\sigma'_{i_1}=\id_{d-1}$.) 
\end{proof}
This finishes the proof of the lemma.
  \end{proof}

\subsection{The proof of Theorem \ref{thm:esistenza}}

Now Theorem \ref{thm:esistenza} is an immediate consequence:

  \begin{proof}[Proof of Theorem \ref{thm:esistenza}]
      The assumption that the Riemann-Hurwitz condition \eqref{eq:RH_cond2}
    is satisfied yields that the $n$-tuple $(e_1,\ldots,e_n)$ is $(d,n)$-admissible.  Hence,  by Lemma \ref{lemma:K2} there are cycles $\sigma_1,\ldots,\sigma_{n}$ in $\Sym(d)$ with $\length \sigma_i= e_i$ 
such that
$\sigma_{n}\cdots\sigma_1=\id_d$ and the subgroup generated by $\sigma_1,\ldots,\sigma_{n}$ is transitive.  
By the
Riemann Existence Theorem there exists a compact Riemann surface $C$ of genus $g$ and a cover $f:C \to \PP^1$ branched at $y_1,\ldots, y_n$ with ramification profile $[e_i,1^{d-e_i}]$ over $y_i$. 
\end{proof}

  \end{document}